\documentclass[12pt]{amsart}
\usepackage{amssymb,amsthm,graphicx,caption,mwe,float,quiver,adjustbox}%

\newtheorem{thm}{Theorem}
\newtheorem{theorem}{Theorem}[section]
\newtheorem{corollary}[theorem]{Corollary}
\newtheorem{proposition}[theorem]{Proposition}
\newtheorem{lemma}[theorem]{Lemma}

\title [Covering  by Centralizers]{Covering by Centralizers}
\author[Lewis]{Mark L. Lewis}
\address{Department of Mathematical Sciences, Kent State University, Kent, OH  44242; lewis@math.kent.edu}
\author[McCulloch]{Ryan McCulloch}
\address{Department of Mathematics and Statistics, Binghamton University, Binghamton, NY 13902; rmccullo1985@gmail.com}
\date{}

\begin{document}
	
\begin{abstract}
In this paper, we consider covers of finite groups by centralizers of elements.  We show that the set of centralizers that are maximal under the partial ordering form a cover of the group.  We also show that the set of centralizers that are minimal under the partial ordering form a cover of the group.  We show for $F$-groups that are nonabelian $p$-groups that the number of  distinct nontrivial centralizers is congruent to $1$ modulo $p$.
\end{abstract}


\maketitle    
	%\begin{document}
	
\section{Introduction}

In this paper, all groups are finite.  Given a group $G$, a {\it cover} for $G$ is a set of proper subgroups whose union is all of $G$.  A cover is called {\it irredundant} if all proper subsets of the cover are not covers of $G$.  While there has a push in current research to determine the irredundant covers of smallest size (see \cite{Gar} and {\cite{GKS}), another strain in current research is to look at covers where all the elements in the cover satisfy a certain property (see \cite{barg}, \cite{mason} and \cite{sun}).  We recommend the expository paper \cite{cover-exp} for a description of the work on covers that is going on.  

In this paper, we consider covers where the cover consists of centralizers of noncentral elements or consists of closely related subgroups.  We begin by considering covers that consist entirely of centralizers of noncentral elements.  Fix a nonabelian group $G$.  Set 
$$\mathcal {C} (G) = \{ C_G (g) \mid g \in G \setminus Z(G) \}.$$  
Since every element of $G$ either lies in its own centralizer or is in the center, $\mathcal{C} (G)$ forms a cover of $G$.  We can use containment to define a partial ordering on the centralizers.  Using this partial ordering, we can obtain two subsets of the set of all centralizers that are covers.

\begin{thm}\label{Intro thm 1}
Let $G$ be a nonabelian group.  Then the following are true:
\begin{enumerate}
\item The set of maximal elements of $\mathcal {C} (G)$ under containment form a cover of $G$.
\item The set of minimal elements of $\mathcal {C} (G)$ under containment form a cover of $G$.
\end{enumerate}
\end{thm}

For each element $g \in G \setminus Z(G)$, we define $Z(g) = Z (C_G (g))$ and we say that $Z (g)$ is the {\it center} of $g$.  We set $\mathcal{Z} (G) = \{ Z(g) \mid g \in G \setminus Z (G) \}$.  Since every element of $g$ either lies in its center or in $Z (G)$ and we see that $Z(G) \le Z(g)$  for all $g \in G \setminus Z(G)$, we have $\mathcal{Z} (G)$ is a cover for $G$. Again, we see that containment defines a partial ordering of $\mathcal{Z} (G)$.   Using this partial ordering we obtain the only irredundant cover using element centers.

\begin{thm}\label{Intro thm 2}
If $G$ is a nonabelian group, then the maximal elements of $\mathcal{Z} (G)$ under containment is the only subset of $\mathcal{Z} (G)$ that is an irredundant cover of $G$.     
\end{thm}

Using the maximal centers of elements, we also obtain information about covers by centralizers and those that are irredundant.

\begin{thm} \label{intro max irredundant}
Let $G$ be a nonabelian group and let $\mathfrak{C}$ be a subset of $\mathcal{C} (G)$.  Then the following are true:
\begin{enumerate}
\item $\mathfrak{C}$ is a cover of $G$ if and only if every maximal subgroup of $\mathcal {Z} (G)$ lies in some subgroup of $\mathfrak {C}$.
		%\item If $\mathfrak{C}$ is an irredundant cover, then every subgroup in $\mathfrak{C}$ contains some maximal subgroup of $\mathcal {Z} (G)$.
\item If $\mathfrak{C}$ is an irredundant cover of $G$, then it has size at most the number of maximal subgroups in $\mathcal{Z} (G)$.
\end{enumerate}
\end{thm}

In \cite{centgraph}, we introduced a graph on centralizers that is related to the commuting graph.  We call this graph the centralizer graph and we use $\Gamma_\mathcal{Z} (G)$ to denote it.  We will give the formal definition of this graph and more information in Section \ref{secn:graph}.  We have the following connection between dominating sets in this graph and covers of centralizers.

\begin{thm} \label{intro dom}
Let $G$ be a nonabelian group and let $\mathfrak{C}$ be a subset of $\mathcal{C}(G)$.  Then the following are true:
\begin{enumerate}
\item $\mathfrak{C}$ is a cover of $G$ if and only if $\mathfrak{C}$ is a dominating set for $\Gamma_{\mathcal Z} (G)$.
\item $\mathfrak{C}$ is an irredundant cover of $G$ if and only if $\mathfrak{C}$ is a minimal dominating set for $\Gamma_{\mathcal Z} (G)$.
\end{enumerate}
\end{thm}

A particular type of cover is a partition.  In particular, a {\it partition} of a set $S$ is a set of subsets whose union is $S$ and that intersect trivially pairwise.  A {\it group partition} for a group $G$ is a cover where the intersection of any two subsets is the identity.   We recommend the reader consult \cite{zappa} for more information on group partitions.  An $F$-group is a group where there are no proper containments among the subgroups in $\mathcal {C} (G)$.  We obtain another characterization of $F$-groups.

\begin{thm} \label{intro F-group partition}
Let $G$ be a nonabelian group.  Then $G$ is an $F$-group if and only if $G \setminus Z(G)$ is partitioned by $Z(g) \setminus Z(G)$ as $Z(g)$ runs over the subgroups in $\mathcal{Z} (G)$.	
\end{thm}

For $F$-groups, we are also able to obtain a count on the number of noncentral centralizers when the group is a nonabelian $p$-group.

\begin{thm} \label{intro $F$-group $p$-group}
In a nonabelian $p$-group that is an $F$-group, the number of noncentral centralizers is congruent to $1$ modulo $p$.     
\end{thm}

A CA-group is a group where all of the subgroups in $\mathcal {C} (G)$ are abelian.  It is well known that CA-groups are $F$-groups.  We obtain the following characterization of CA-groups.

\begin{thm}\label{intro thm:CA}
Let $G$ be an $F$-group.  We have $\mathcal {C} (G)$ is an irredundant cover of $G$ if and only if $G$ is a CA-group.
\end{thm}

We want to close this introduction by mentioning that the genesis of this paper came when the authors attended a talk at the 2025 Zassenhaus Group Theory conference by Matthew Klepadlo who reported on work about covering dihedral and permutation groups with centralizers.  (See \cite{zass}.)  We should also mention the paper \cite{HaAm} which also considers various questions regarding covers by centralizers.  We would like to thank Tuval Foguel for reminding us about the paper by Haji and Amiri and to Matthew for inspiring us to work on this problem.

\section {Centers and Centralizers}

%If $g \in G \setminus Z(G)$, we will write $Z (g) = Z (C_G (g))$.  We write $\mathcal {Z} (G) = \{ Z(g) \mid G \in \setminus Z(G) \}$ and $\mathcal{C} (G) = \{ C_G (g) \mid g \in G \setminus Z(G) \}$.  

In this section, we review a number of results that have been proven in our previous papers.  The following lemma has been proved in \cite{max abel}.

\begin{lemma} \label{three}
Let $G$ be a group and suppose $a,b \in G \setminus Z(G)$.
\begin{enumerate}
\item $a \in C_G(b)$ if and only if $Z(a) \le C_G (b)$.
\item $Z(a) \le C_G (b)$ if and only if $Z(b) \le C_G (a)$.
\end{enumerate}
\end{lemma}

The following lemma is partly an immediate observation and partly folklore that we gave a proof of in \cite{max abel}.

\begin{lemma} \label{abel cent}
Let $G$ be a group and suppose $a \in G \setminus Z(G)$.  Then the following are equivalent:
\begin{enumerate}
\item $C_G (a) = Z (a)$.
\item $C_G (a)$ is abelian.
\item $C_G (a)$ is maximal among abelian subgroups of $G$.
\end{enumerate}
\end{lemma}

It is not difficult to see that every element in $G \setminus Z(G)$ lies in some $Z(g)$.  We now show that the centralizers of $G$ are covered by these sets.  This is proved in \cite{centgraph}.

\begin{corollary} \label{centr equal union cents}
Let $G$ be a group, and let $a \in G \setminus Z(G)$.  Then $\displaystyle C_G (a) = \bigcup_{b \in C_G(a) \setminus Z(G)} Z(b)$.
\end{corollary}

The next lemma is key for understanding the relationships between centers and centralizers.  This is proved in \cite{centgraph}. 

\begin{lemma} \label{star 1}
Let $G$ be a group and let $g, h \in G \setminus Z(G)$.  Then 
\begin{enumerate}
\item $C_G (g) \le C_G (h)$ if and only if $Z(h) \le Z(g)$.
\item $C_G (g) = C_G (h)$ if and only if $Z (g) = Z(h)$.
\end{enumerate}
\end{lemma}

Notice that Lemma \ref{star 1} (2) gives us a bijection between $\mathcal{Z} (G)$ and $\mathcal {C} (G)$, and Lemma \ref{star 1} (1) implies that this bijection is containment reversing.  This is written formally in this next corollary.  See \cite{centgraph}.

\begin{corollary}\label{lem:corres}
Let $G$ be a group. 
Then the map $\phi : \mathcal {C} (G) \rightarrow \mathcal {Z} (G)$ given by 
$$\phi(X) = Z (X) = C_G (X)$$
is an order-reversing $1$-to-$1$ correspondence with $\phi^{-1} : \mathcal {Z} (G) \rightarrow \mathcal {C} (G)$ given by $$\phi^{-1} (Y) = {C}_G (Y).$$
\end{corollary}

We note that $\mathcal {Z} (G)$ and $\mathcal {C} (G)$ are not usually lattices, but we obtain the following for intersections.  The following is proved in \cite{centgraph}.

\begin{lemma} \label{intersection}
Let $G$ be a group and let $g, h \in G \setminus Z(G)$.   Then either $Z(g) \cap Z(h) = Z(G)$ or $(Z(g) \cap Z(h)) = \prod Z(c)$ where $c$ runs over the elements in $(Z (g) \cap Z(h)) \setminus Z(G)$.
\end{lemma}

\section {Covers by centers and centralizers}

In this section, we prove the results regarding the covers by centralizers and centers of elements.

%In this section, we define covers of groups and find some cover by centers and centralizers.  A {\it cover} of a group $G$ is a set of subgroups of $G$ such that every element of $G$ lies in some subgroup of the cover.  Since every element $g \in G \setminus Z(G)$ lies in some subgroup in both $\mathcal {C} (G)$ and $\mathcal {C} (G)$ form covers of $G$.  In fact, we do not need all of $\mathcal {C} (G)$ to obtain a cover.  We start with a sufficient condition for a subset of $\mathcal{C} (G)$ to be a cover of $G$.

\begin{proposition}\label{prop:cover_contain_z}
Let $G$ be a nonabelian group, and let $\mathfrak{C}$ be a subset of $\mathcal{C} (G)$.  Then $\mathfrak{C}$ is a cover of $G$ if and only if for each $x \in G \setminus {Z}(G)$, the center ${Z}(x)$ is contained in some member of $\mathfrak{C}$.
\end{proposition}

\begin{proof}
($\Leftarrow$) A centralizer contains ${Z}(G)$, and so for each $z \in {Z}(G)$, $z$ is in every member of $\mathfrak{C}$.  Note that for each element $x \in G \setminus {Z}(G)$, we have $x \in {Z}(x)$, and hence, $x$ is in a member of $\mathfrak{C}$.  Thus, $\bigcup \mathfrak{C} = G$ and $\mathfrak{C}$ is a cover.
	
($\Rightarrow$) Suppose $\mathfrak{C}$ is a cover of $G$.  For each element $x \in G \setminus {Z}(G)$, we have $x$ is in some ${C}_G (a) \in \mathfrak{C}$.  If follows by Lemma \ref{three} that for each $x \in G \setminus {Z}(G)$ that ${Z}(x)$ is contained in some ${C}_G(a) \in \mathfrak{C}$.  
\end{proof}

%We now find three particular sets of centers and centralizers that are covers.  When we talk about maximal and minimal subgroups in $\mathcal {C} (G)$ and $\mathcal {Z} (G)$, we are talking in terms of containment.  
Note that Lemma \ref{star 1} (1) implies that $C_G (x)$ is maximal in $\mathcal{C} (G)$ if and only if $Z(x)$ is minimal in $\mathcal {Z} (G)$, and $C_G (x)$ is minimal in $\mathcal{C} (G)$ if and only if $Z(x)$ is maximal in $\mathcal {Z} (G)$.

The following theorem includes Theorem \ref{Intro thm 1} and more.

\begin{theorem}
Let $G$ be a nonabelian group.  Then the following are true:
\begin{enumerate}
\item The maximal elements in $\mathcal {C} (G)$ form a cover of $G$.
\item The maximal elements in $\mathcal {Z} (G)$ form a cover of $G$.
\item The minimal elements in $\mathcal {C} (G)$ form a cover of $G$.
\end{enumerate}
\end{theorem}

\begin{proof}
We know that $\mathcal{C} (G)$ gives a cover of $G$.   Since every subgroup in $\mathcal{C} (G)$ lies in a subgroup that is maximal in $\mathcal{C} (G)$, the maximal subgroups in $\mathcal {C} (G)$ are sufficient for a cover.  Similarly, we know that $\mathcal{Z} (G)$ gives a cover of $G$, and so, the maximal subgroups of $\mathcal {Z} (G)$ are sufficient to form a cover.  Finally, each maximal subgroup of $\mathcal{Z} (G)$ is the center of a minimal subgroup of $\mathcal{C} (G)$, so the minimal elements of $\mathcal{C} (G)$ will form a cover.
\end{proof}

We next make an observation about centralizers and maximal centers.

\begin{lemma} \label{maximal containnment}
Let $G$ be a nonabelian group and let $x \in G \setminus Z (G)$.  Then $C_G (x)$ contains some maximal subgroup $Z(y)$ among the subgroups in $\mathcal{Z} (G)$.
\end{lemma}

\begin{proof}
We know that $C_G (x)$ must contain a minimal subgroup $C_G (y)$ in $\mathcal {C} (G)$.  This implies that $Z(y)$ is maximal in $\mathcal{Z} (G)$.  We see that $Z(y) \le C_G (y) \le C_G (x)$.  This proves the lemma.
\end{proof}

We now consider a useful equivalence relation.  We define this equivalence relation on $G \setminus Z(G)$.  We say $g \bowtie h \in G \setminus Z(G)$ if $Z(g) = Z(h)$.  It is not difficult to see that $\bowtie$ is an equivalence relation on $G \setminus Z(G)$.  If we set $Z^* (g) = \{ h \in G \setminus Z(G) \mid Z(h) = Z(g) \}$, then $Z^* (g)$ is the equivalence class of $g$ under this equivalence relation.   

We can define another equivalence relation on $G \setminus Z(G)$ by $g \sim h$ if $C_G (g) = C_G (h)$.  By Lemma \ref{star 1}, $\bowtie$ is exactly the equivalence relation $\sim$, and so $Z^* (g) = \{h \in G \setminus Z(G) \mid C_G (h) = C_G (g)\}$.  Observe from the definition that $Z^* (g) \subseteq Z(g) \setminus Z(G)$.

\begin{lemma} \label{star 2}
Let $G$ be a nonabelian group and $g \in G \setminus Z(G)$.  If $g_1, \dots, g_t \in G \setminus Z(G)$ such that $C_G (g_1), \dots, C_G (g_t)$ are distinct and comprise all of the elements of ${\mathcal C} (G)$ that properly contain $C_G (g)$, then $\displaystyle Z(g) \setminus Z(G) = Z^* (g) \cup (\bigcup_{i=1}^t Z^* (g_i))$.   In particular, $\displaystyle \bigcup_{i=1}^t Z^* (g_i)$ is a proper subset of $Z(g) \setminus Z(G)$. In addition, if $h \in G \setminus Z(G)$, then the following are equivalent:
\begin{enumerate}
\item $h \in Z(g)$.
\item $Z^* (h) \subseteq Z(g)$.
\item $Z(h) \le Z(g)$.
\item $C_G (g) \le C_G (h)$.
\end{enumerate}
\end{lemma}

\begin{proof}
Since $C_G (g) < C_G (g_i)$, we have $Z^* (g_i) \subseteq Z(g_i) \setminus Z(G) \subseteq Z(g) \setminus Z(G)$ by Lemma \ref{star 1} for each $i$.  It follows that $\displaystyle \bigcup_{i=1}^t Z^* (g_i) \subseteq Z(g) \setminus Z(G))$.  Since $Z^* (g) \subseteq Z(g) \setminus Z (G)$, we have $\displaystyle Z^* (g) \cup (\bigcup_{i=1}^t Z^* (g_i)) \subseteq Z(g) \setminus Z(G)$.
	
On the other hand, if $h \in Z(g) \setminus Z(G))$, then $h$ centralizes $C_G (g)$.  We see that $C_G (g)$ centralizes $h$; so $C_G (g) \le C_G (h)$.  If $C_G (g) = C_G (h)$, then $h \in Z^* (g)$.  Thus, we may assume $C_G (g) < C_G (h)$.  Now, $C_G (h)$ is an element of ${\mathcal C} (G)$ that contains $C_G (g)$.   Since $h \not\in Z(G)$, it follows that $C_G (h) = C_G (g_i)$ for a unique $i$ with $1 \le i \le t$.  This implies $h \in Z^* (g_i)$, and so, $\displaystyle h \in \bigcup_{i=1}^t Z^* (g_i)$.  We conclude that $\displaystyle Z(g) \setminus Z(G)) = Z^* (g) \cup (\bigcup_{i=1}^t Z^* (g_i))$.   Since the $Z^*$'s are disjoint and $Z^* (g)$ is not empty, this implies that  the conclusion that $\displaystyle \bigcup_{i=1}^t Z^* (g_i)$ is proper in $Z(g) \setminus Z(G)$ is now immediate. 
	
We know that $Z(h) \le Z(g)$ is equivalent to $C_G (h) \le C_G (g)$.  Suppose $C_G (g) \le C_G (h)$.  If $C_G (g) = C_G (h)$, then $Z^* (g) = Z^* (h)$.  Also, when $C_G (g) < C_G (h)$, we have shown that $Z^* (h) \subseteq Z (g)$.  This shows that $Z^* (h) \subseteq Z^* (g)$ in both cases.
	
It suffices to show that if $Z^* (h) \subseteq Z (g)$, then $Z(h) \le Z(g)$.  We show the contrapositive.  Suppose $Z (h) \not\le Z (g)$.  This implies that $h \not\in Z(g)$ since by Lemma \ref{intersection}, we have $Z(h) \cap Z(g)$ is the product of $Z (c)$ where $c$ runs over the elements in $Z(h) \cap Z(g)$.  Since $h$ lies in $Z^* (h)$, this implies that $Z^* (h)$ is not contained in $Z(g)$, and we have proved the contrapositive.
	
Suppose $h \in Z(g)$.  Then $\displaystyle h \in Z^* (g) \cup (\bigcup_{i=1}^t Z^* (g_i))$.  Either $h \in Z^* (g)$ or $h \in Z^* (g_i)$ for some $i$.  In either case, $Z (h) \le Z(g)$.  Conversely, if $Z(h) \le Z(g)$, then $h \in Z(h)$ implies $h \in Z(g)$.  This shows that all of the equivalences hold.
\end{proof}

Note that $C_G (g)$ being minimal in ${\mathcal C} (G)$ is equivalent to $Z (g)$ being maximal in ${\mathcal Z} (G)$.  Similarly, $C_G (g)$ being maximal in ${\mathcal C} (G)$ is equivalent to $Z (g)$ being minimal in ${\mathcal Z} (G)$.  The following corollary is an immediate consequence of Lemma \ref{star 2}.

\begin{corollary} \label{maximal Z}
Let $G$ be a nonabelian group and $g \in G \setminus Z(G)$.  If $Z(g)$ is maximal in ${\mathcal {Z} (G)}$, then $Z(g)$ is the only subgroup in ${\mathcal {Z} (G)}$ that contains $Z^* (g)$.
\end{corollary}

%We now apply this equivalence relation to covers.  A cover is {\it irredundant} if no subset of the cover is also cover.  Using the $Z^*$'s we will see that we can determine that there is a unique irredundant cover of centers, and this allows to determine which sets of centers will be covers.

Notice that this next theorem includes Theorem \ref{Intro thm 2}.

\begin{theorem} \label{max centers irredundant}
Let $G$ be a group.  Then the maximal elements of $\mathcal {Z} (G)$ form an irredundant cover of $G$.  In fact, a subset of $\mathcal {Z} (G)$ is a cover if and only if it contains the maximal elements of $\mathcal {Z} (G)$.
\end{theorem}

\begin{proof}
Let $x_1, \dots, x_n$ be representatives of the distinct maximal elements of $\mathcal {Z} (G)$.  Since $Z(x_i)$ is the only subgroup in $\mathcal{Z} (G)$ that contains $Z^* (x_i)$ (by Corollary \ref{maximal Z}) and the $Z^*$'s partition $G \setminus Z(G)$, it follows that for a subset to cover $G$ it must contain all of the $Z(x_i)$'s which are all of the maximal subgroups in $\mathcal {Z} (G)$.  Furthermore, since none of these can be missing for the set of maximal subgroups in $\mathcal {Z} (G)$ to be a cover, it must be irredundant.  
\end{proof}

With this, we can obtain Theorem \ref{intro max irredundant} which is  a refinement of Proposition \ref{prop:cover_contain_z}.

%\begin{theorem} \label{max irredundant}
%	Let $G$ be a group and let $\mathfrak{C}$ be a subset of $\mathcal{C} (G)$.  Then the following are true:
%	\begin{enumerate}
	%	\item $\mathfrak{C}$ is a cover of $G$ if and only if every maximal subgroup of $\mathcal {Z} (G)$ lies in some subgroup of $\mathfrak {C}$.
		%\item If $\mathfrak{C}$ is an irredundant cover, then every subgroup in $\mathfrak{C}$ contains some maximal subgroup of $\mathcal {Z} (G)$.
%		\item An irredundant cover of $G$ of centralizers has size at most the number of maximal subgroups in $\mathcal{Z} (G)$.
%	\end{enumerate}
%\end{theorem}

\begin{proof}[Proof of Theorem \ref{intro max irredundant}]
Suppose $\mathfrak{C}$ is a cover of $G$.  By Proposition \ref{prop:cover_contain_z}, we know that every maximal subgroup in $\mathcal {Z} (G)$ lies in some subgroup of $\mathfrak {C}$.  Conversely, suppose every maximal subgroup of $\mathcal {Z} (G)$ lies in some subgroup of $\mathfrak {C}$.  Let $Z(x)$ be some center in $\mathcal{Z} (G)$.  We know $Z (x)$ lies in some maximal subgroup in $\mathcal {Z} (G)$ and so it lies some subgroup of $\mathfrak {C}$.  We see that $\mathfrak{C}$ is a cover of $G$ by Proposition \ref{prop:cover_contain_z}.  This proves (1).
	
	%We prove the contrapositive.  Suppose some subgroup $C$ in $\mathfrak{C}$ does not contain a maximal subgroup in $\mathcal {Z} (G)$.  By (1), every maximal subgroup in ${\mathcal{Z}} (G)$ lies in some subgroup of $\mathfrak{C}$ and so it must lie in $\mathfrak{C} \setminus \{ C\}$.  By (1), $\mathfrak{C} \setminus \{ C\}$ is a cover, and therefore, $\mathfrak{C}$ is not irredundant.  Hence, (2) is proved.  Notice that (3) follows immediately from (2). 
	%By Lemma \ref{maximal containnment}, every subgroup in $\mathfrak{C}$ contains some maximal subgroup of $\mathcal {Z} (G)$.
Suppose $\mathfrak{C}$ is an irredundant cover.  Let $Z_1, \dots, Z_m$ be the maximal subgroups in $\mathcal{Z} (G)$.  By (1), for each $i$ with $1 \le i \le m$, we can find some $C_i \in \mathfrak C$ so that $Z_i \le C_i$.  It follows that $\{ C_1, \dots, C_m \} \subseteq \mathfrak{C}$ and by (1), $\{ C_1, \dots, C_m \}$ is a cover.  Since $\mathfrak{C}$ is irredundant, we have $\mathfrak{C} = \{ C_1, \dots, C_m \}$.  Because the $C_i$'s are not necessarily distinct, we obtain $|\mathfrak{C}| \le m$ which proves (2).
\end{proof}

\section{covers by centralizers and the centralizer graph}\label{secn:graph}

In this section, we consider the relationship between covers consisting of element centralizers and the centralizer graph.  Let $G$ be a group.  We define the centralizer graph $\Gamma_{\mathcal{Z}} (G)$ to be the graph whose vertex set is $\mathcal{Z} (G)$ and there is an edge between $Z(g)$ and $Z(h)$ if $Z(h) \le C_G (h)$.  We note that we use the bijection between $\mathcal {Z} (G)$ and $\mathcal {C} (G)$ since sometimes we will want to use $\mathcal {Z} (G)$ to label the vertices and other times (even within the same result) we will want to use $\mathcal {C} (G)$ to label the vertices.

A set of vertices $S$ in a graph $\Gamma$ is said to be a \textit{dominating set for $\Gamma$} if every vertex of $\Gamma$ is either in $S$ or is adjacent to a vertex in $S$.  A dominating set is \textit{minimal} if it contains no proper subset that is a dominating set.

We first prove Theorem \ref{intro dom}.

%\begin{theorem}
%Let $G$ be a group.  If the minimal subgroups in $\mathcal {Z} (G)$ are not adjacent in $\Gamma_{\mathcal Z} (G)$, then the maximal subgroups in $\mathcal {C} (G)$ are an irredundant cover of $G$.
%\end{theorem}

\begin{proof}[Proof of Theorem \ref{intro dom}]
By Proposition \ref{prop:cover_contain_z}, $\mathfrak{C}$ is a cover of $G$ if and only if every member of $\mathcal{Z}(G)$ is contained in a member of $\mathfrak{C}$.  This holds if and only if $\mathfrak{C}$ is a dominating set for $\Gamma_{\mathcal Z} (G)$.  To see this, let $Z(x) \in \mathcal{Z}(G)$ be arbitrary, and note that $Z(x)$ is contained in some member of $\mathfrak{C}$ if and only if either $C_G(x) \in \mathfrak{C}$ or $C_G(x) \notin \mathfrak{C}$ and $Z(x) \leq C_G(y)$ for some $C_G(y) \in \mathfrak{C}$; in other words, if and only if $\mathfrak{C}$ is a dominating set for $\Gamma_{\mathcal Z} (G)$.

Clearly, irredundancy holds if and only if $\mathfrak{C}$ is minimal as a dominating set.
%
%To prove the converse, we also prove the contrapositive.  I.e., we show that if some pair of minimal subgroups in $\mathcal {Z} (G)$ are adjacent in $\Gamma_{\mathcal Z} (G)$,  then the maximal subgroups in $\mathcal {C} (G)$ are not irredundant.  Let $Z(x_1), \dots, Z(x_n)$ be the minimal subgroups in $\mathcal {Z} (G)$.  Without loss of generality, suppose that $Z(x_1)$ and $Z(x_2)$ are adjacent in $\Gamma_{\mathcal Z} (G)$.  
\end{proof}

We obtain the following as a corollary.

\begin{theorem} \label{not adj}
Let $G$ be a group.  If the minimal subgroups in $\mathcal {Z} (G)$ are not adjacent in $\Gamma_{\mathcal Z} (G)$, then the maximal subgroups in $\mathcal {C} (G)$ are an irredundant cover of $G$.
\end{theorem}

\begin{proof}
If the minimal subgroups in $\mathcal{Z}(G)$ are not adjacent in $\Gamma_{\mathcal{Z}}(G)$, then the maximal subgroups in $\mathcal{C}(G)$ are not adjacent in $\Gamma_{\mathcal{Z}}(G)$.  We know the maximal subgroups in $\mathcal{C}(G)$ are a cover of $G$ and hence a dominating set for $\Gamma_{\mathcal{Z}}(G)$ by Theorem \ref{intro dom}.  If these vertices are not pairwise adjacent in $\Gamma_{\mathcal{Z}}(G)$ then this is a minimal dominating set, and hence an irredundant cover by Theorem \ref{intro dom}.
\end{proof}

\section{$F$-groups}

Expanding on the work of It\^o in \cite{ito}, Rebmann in \cite{reb} defines a group $G$ to be an {\it $F$-group} if for all $x,y \in G \setminus Z(G)$ the condition $C_G(x) \le C_G (y)$ implies $C_G (x) = C_G (y)$.  We note that every semi extraspecial group is an $F$-group.  Recall that a $p$-group $G$ is semi extraspecial if every subgroup $N$ that is maximal in $Z(G)$ satisfies that $G/N$ is extraspecial.  %It is not difficult to see that every CA-group will be an $F$-group.  However, there are many $F$-groups that are not CA-groups.  In particular, 
%every ultraspecial group is an $F$-group, and for $p^{3n}$ with $n \ge 3$ and $p$ any prime, there  will be only one CA-group that is ultraspecial, but many, many ultraspecial groups that are not CA-groups.

%The following lemma shows that being an $F$-group is equivalent to ${\mathcal Z} (G)$ partitions the elements of $G \setminus Z(G)$.

We start with a lemma that shows that there are equivalent versions of the definition of $F$-group that can be expressed in terms of the centers of elements.

\begin{lemma} \label{part}
Let $G$ be a group.  Then the following are equivalent:
\begin{enumerate}
\item $G$ is an $F$-group.
\item For all $a,b \in G \setminus Z(G)$, the element $b \in Z (a)$ implies $Z (a) = Z (b)$.
\item For all $a,b \in G \setminus Z(G)$, either $Z(a) = Z(b)$ or $Z (a) \cap Z(b) = Z(G)$.
\item For all $a,b \in G \setminus Z(G)$, $Z (b) \le Z (a)$ implies $Z (a) = Z (b)$.
\end{enumerate}
\end{lemma}

\begin{proof}
Suppose (1).  Consider $a,b \in G \setminus Z(G)$.  Suppose $b \in Z(a)$.  By Lemma \ref{star 2}, we have $Z(b) \le Z(a)$.  Then by Lemma \ref{star 1} (1), we have $C_G (a) \le C_G (b)$.  Since $G$ is an $F$-group, this implies $C_G (a) = C_G (b)$.  Using Lemma \ref{star 1} (2), we obtain $Z (a) = Z (b)$.  This proves (2).
	
Suppose (2).  I.e., suppose that $b \in Z(a)$ implies $Z(b) = Z(a)$.  Suppose $b \in G \setminus Z(G)$ so that $Z(b) \ne Z(a)$.  By Lemma \ref{star 1} (2), this implies that $C_G (b) \ne C_G (a)$.  Since $G$ is an $F$-group, this implies that $C_G (b)$ is not contained in $C_G (a)$ and $C_G (a)$ is not contained in $C_G (b)$.  By Lemma \ref{star 1} (1), we have $Z(b)$ is not contained in $Z (a)$ and $Z(a)$ is not contained in $Z (b)$.  In particular, $Z(b) \cap Z(a)$ is properly contained in both $Z(a)$ and $Z (b)$.  If $Z(a) \cap Z(b) > Z(G)$, then there exists $c \in Z (a) \cap Z (b) \setminus Z (G)$.  Observe that $Z (c) < Z (a)$ and by Lemma \ref{star 1}, we have $C_G (a) < C_G (c)$, which is a contradiction.  Thus (3) holds.
	
Suppose (3).  I.e., suppose either $Z(a) = Z(b)$ or $Z(a) \cap Z (b) = Z (G)$.   Suppose  $x,y \in G \setminus Z(G)$, so that $C_G (x) \le C_G (y)$.  By Lemma \ref{star 1} (1), we have $Z (y) \le Z(x)$.  Observe that $Z(x) \cap Z(y) = Z(y) > Z(G)$, so we must have $Z(x) = Z(y)$, and by Lemma \ref{star 1}  (2), $C_G (x) = C_G (y)$.  Thus, $G$ is an $F$-group, and (1) holds. 
	
Suppose (1).  Suppose $Z(b) \le Z(a)$.  By Lemma \ref{star 1} (1), we have $C_G (a) \le G_G (b)$ and thus, $C_G (a) = C_G (b)$.  Applying Lemma \ref{star 1} (2), we conclude that $Z (a) = Z(b)$ which is (4).  Conversely, assume (4).  Suppose $C_G (x) \le C_G (y)$.  By Lemma \ref{star 1} (1), $Z(y) \le Z(x)$, and this yields $Z (x) = Z(y)$.  In light of Lemma \ref{star 1} (2), we conclude that $C_G (x) = C_G (y)$, and hence, $G$ is an $F$-group.  This proves (1).
\end{proof}

We now have several corollaries of Lemma \ref{part}.  This first one shows that when $b \not\in Z(a)$, then $Z(b) \cap Z(a) = Z(G)$.

\begin{corollary} \label{twob}
Let $G$ be an $F$-group.  If $b \in G \setminus Z(a)$, then $Z (a) \cap Z(b) = Z(G)$.
\end{corollary}

\begin{proof}
Since $b \not\in Z(a)$, we know from Lemma \ref{part} (2) that $Z (a) \ne Z(b)$.  Now, applying Lemma \ref{part} (3), we have $Z(a) \cap Z (b) = Z(G)$.
\end{proof}

An immediate corollary of Lemma \ref{part} is that $G$ is an $F$-group exactly when the subgroups in either $\mathcal{C} (G)$ or $\mathcal{Z} (G)$ are all both maximal and minimal.

\begin{corollary}\label{maxmin}
Let $G$ be a group.  The following are equivalent:
\begin{enumerate}
\item $G$ is an $F$-group.
\item every subgroup in $\mathcal {C} (G)$ is both maximal and minimal in $\mathcal{C} (G)$.
\item every subgroup in $\mathcal {Z} (G)$ is both maximal and minimal in $\mathcal{Z} (G)$.
\end{enumerate}
\end{corollary}

%Note that Corollary \ref{??} implies that $C_G (a) \setminus Z (G)$ is partitioned by a subset of ${\mathcal Z} (G)$ for each $a \in G \setminus Z(G)$.  For $F$-groups, we can strengthen this.

\begin{corollary} \label{threea}
Let $G$ be an $F$-group.  If $a, b \in G \setminus Z(G)$, then either $C_G (a) \cap Z(b) = Z(G)$ or $Z(b) \le C_G(a)$.
\end{corollary}

\begin{proof}
We know $Z(G) \le C_G (a) \cap Z(b)$.  Suppose $Z(G) < C_G(a) \cap Z(b)$.  Then there exists $c \in C_G (a) \cap Z(b) \setminus Z(G)$.  Since $c \in Z(b)$, we know by Lemma \ref{part} (2) that $Z(c) = Z(b)$.  On the other hand, because $c \in C_G (a)$, we have by Lemma \ref{three} that $Z(c) \le C_G (a)$.  This proves the result.
\end{proof}

We now get another perspective on the relationship between $\mathcal{Z} (G)$ and the $Z^*$'s.  Notice that this also formalizes the relationship between maximal subgroups in $\mathcal {C} (G)$ and minimal subgroups in $\mathcal{Z} (G)$.

\begin{lemma} \label{maximal}
Let $G$ be a group and let $g \in G \setminus Z(G)$.   The following are equivalent:
\begin{enumerate}
\item $C_G(g)$ is maximal in ${\mathcal C} (G)$ %if and only if 
\item $Z(g)$ is minimal in $\mathcal{Z} (G)$
\item $Z^* (g) = Z(g) \setminus Z(G)$.
		%\item Suppose $C_G (g)$ is not maximal in ${\mathcal C} (G)$ and suppose $g_1, \dots, g_t \in G \setminus Z(G)$ such that $C_G (g_1), \dots, C_G (g_t)$ is an irredundant and complete list of the elements of ${\mathcal C} (G)$ that properly contain $C_G (g)$.  Then $Z^*(g) = Z(g) \setminus (Z(G) \cup_{i=1}^t Z^* (g_i))$.
\end{enumerate}
\end{lemma}

\begin{proof}
	%The first statement 
This is immediate from Lemma \ref{star 2} by taking $t = 0$.  
%Suppose now that $C_G (g)$ is not maximal in ${\mathcal C} (G)$ and suppose $g_1, \dots, g_t \in G \setminus Z(G)$ such that $C_G (g_1), \dots, C_G (g_t)$ is an irredundant and complete list of the elements of ${\mathcal C} (G)$ that properly contain $C_G (g)$.  Applying Lemma \ref{star 2}, we have $Z(g) \setminus (Z(G) \cup_{i=1}^t Z^* (g_i)) = (Z(g) \setminus Z(G)) \setminus (\cup_{i=1}^t Z^* (g_i)) = (Z^* (g) \cup_{i=1}^t Z^*(g_i)) \setminus \cup_{i=1}^t Z^*(g_i) = Z^* (g)$.
\end{proof}

Notice that condition (2) in Lemma \ref{part} says that $G \setminus Z(G)$ is partitioned by the sets $\{ Z \setminus Z(G) \mid Z \in {\mathcal Z} (G) \}$.  We now prove Theorem \ref{intro F-group partition}.

%\begin{corollary} \label{F-group partition}
%	Let $G$ be a group.  Then $G$ is an $F$-group if and only if $G \setminus Z(G)$ is partitioned by $Z(g) \setminus Z(G)$ as $Z(g)$ runs over the subgroups in $\mathcal{Z} (G)$.	
%\end{corollary}

\begin{proof}[Proof of Theorem \ref{intro F-group partition}]
Since $G$ is an $F$-group, we have by Corollary \ref{maxmin} that every center in $\mathcal {Z} (G)$ is minimal.  Using Lemma \ref{maximal}, we see that $Z (x) = Z^* (x) \cup Z(G)$ for every $g \in G \setminus Z(G)$.  Since the $Z^*$'s form a partition, the result follows.   Conversely, suppose that $G \setminus Z(G)$ is partitioned by $Z(g) \setminus Z(G)$ as $Z(g)$ runs over the subgroups in $\mathcal{Z} (G)$.	  It follows that if $x, y \in G \setminus Z(G)$ with $Z(x) \neq Z(y)$, then $Z(x) \cap Z(y) = Z(G)$.  By Lemma \ref{part}, $G$ is an $F$-group.
\end{proof}

We obtain the following characterization of nonabelian centralizers in $F$-groups.

\begin{lemma} \label{threeb}
Let $G$ be an $F$-group and fix $a \in G \setminus Z(G)$.  Then $C_G (a)$ is nonabelian if and only if there exists some $b \in G \setminus C_G (a)$ such that $Z(G) < C_G (a) \cap C_G (b)$.
\end{lemma}

\begin{proof}
Suppose that $C_G (a)$ is nonabelian.  This implies $Z (a) < C_G (a)$.  Hence, there exists $c \in C_G(a) \setminus Z(a)$.  By Corollary \ref{twob}, we know that $Z(a) \cap Z(c) = Z(G)$.  If $C_G (c) \le C_G(a)$, then $Z(a) \le Z(c)$ which is a contradiction.  Thus, there exists $b \in C_G (c) \setminus C_G (a)$.  Now, $c \in C_G(b)$, so $c \in (C_G (a) \cap C_G (b)) \setminus Z(G)$.
	
Conversely, assume $Z(G) < C_G(a) \cap C_G(b)$ for some $b \in G \setminus C_G (a)$.  Suppose that $C_G (a)$ is abelian.  This implies $C_G (a) = Z(a)$, and so, $C_G (a) \cap C_G (b) = Z(a) \cap C_G (b)$.  We obtain $Z(G) < Z(a) \cap C_G (b)$.  By Corollary \ref{threea}, we have $Z(a) \le C_G(b)$.  We deduce that $a \in C_G(b)$ and so $b \in C_G (a)$ which is a contradiction.  Therefore, $C_G(a)$ is not abelian.
\end{proof}

%In \cite{ito}, It\^o proved that groups having only two class sizes are $F$-groups.  Since s.e.s. groups have two conjugacy class sizes, we may apply our results for $F$-groups to s.e.s. groups.  Verardi proves that if $G$ is an s.e.s. group and $C_G (a)$ is abelian, then $C_G (a) \cap C_G (b) = G'$ for all $b \in G \setminus C_G (a)$, so Lemma \ref{threeb} gives a different proof of the converse.

%We are particularly interested in s.e.s. groups and note that since they have only two conjugacy class sizes they satisfy condition (*).  In particular, an s.e.s. group $G$ will have that $G \setminus Z(G)$ is partitioned by $\{ Z \setminus Z(G) \mid Z \in {\mathcal Z} (G) \}$.  We expect that there are other groups for which this set forms a partition and ask what can be said about such groups.  At this time, we do not have an answer for this question.

Following \cite{DHJ}, we say that a group $G$ is a CA-group if $C_G (g)$ is abelian for all $g \in G \setminus Z(G)$.  It is not difficult to see that every CA-group will be an $F$-group.  However, there are many $F$-groups that are not CA-groups.  Recall that Suzuki has classified the groups where the centralizers of all nonidentity elements are abelian.  In particular, it is shown in \cite{bsw} and \cite{suz} that if $G$ is a nonabelian group where all the centralizers are abelian, then either $G$ is ${\rm PSL} (2,2^m)$ for some integer $m \ge 2$ or $G$ is a Frobenius group whose Frobenius kernel and Frobenius complement are abelian.

%start here

%\begin{theorem}\label{thm:CA}
%	Let $G$ be an $F$-group.  We have $\mathcal {C} (G)$ is an irredundant cover of $G$ if and only if $G$ is a CA-group.
%\end{theorem}

\begin{proof} [Proof of Theorem \ref{intro thm:CA}]
($\Leftarrow$) In a CA-group, we know by Lemma \ref{abel cent} that $C_G (a) = Z(a)$ for all $a \in G \setminus Z(G)$.  We also know that $G$ is an $F$-group, so every center in maximal in $\mathcal{Z} (G)$.  Using Theorem \ref{max centers irredundant}, $\mathcal {Z} (G) = \mathcal {C} (G)$ is an irredundant cover of $G$.  %If $G$ is a CA-group, then all members of $\mathcal {C} (G)$ are abelian, and so for each $A \in \mathcal {C} (G)$, the set $\mathcal {C}(G) - \{ A \}$ would have no member containing ${Z}(A) = A$.  It follows by Proposition \ref{prop:cover_contain_z} that $\mathcal {C}(G)$ is irredundant.
	
($\Rightarrow$) We prove the contrapositive.  Suppose $G$ is not a CA-group.  Let $a_1, \dots, a_n \in G \setminus Z(G)$ be chosen so that $C_G(a_1), \dots, C_G (a_n)$ are the distinct subgroups in $\mathcal {C}  (G)$ and hence, $Z(a_1), Z(a_2), \dots, Z (a_n)$ are the distinct subgroups in $\mathcal {Z} (G)$.  Since $G$ is not a CA-group, there is some $i$ so that $C_G (a_i)$ is not abelian.  Without loss of generality, we may assume $C_G (a_1)$ is not abelian.  Thus, there is an element $b \in C_G(a_1) \setminus Z(a_1)$.  We know that $b \in Z^* (a_i)$ for some $i$ with $2 \le i \le n$.  Again, without loss of generality, we may assume $b \in Z^* (a_n)$.  This implies that $Z^* (b) = Z^* (a_n) \subseteq C_G (a_1)$.  It follows that $Z (a_n) \le Z(a_1)$.  Let $C^\# = \mathcal {C} (G) \setminus \{ C_G (a_n) \}$.   Observe that for $2 \le j \le n-1$, we have $Z (a_j) \le C_G (a_j) \in C^\#$.  By Proposition \ref{prop:cover_contain_z}, we see that $C^\#$ is a cover of $G$.  This proves that $\mathcal {C} (G)$ is not irredundant. 
	%Suppose $\mathcal {C}(G)$ is irredundant and by way of contradiction suppose $G$ is not a CA-group.  Let ${C}_G(a)$ be a proper centralizer that is nonabelian.  Let $b \in {C}_G(a) \setminus {Z}(a)$.  By Lemma \ref{three} (1), ${Z}(b) \subseteq {C}_G(a)$.  Since $b \notin {Z}(a)$, we have ${Z}(a) \neq {Z}(b)$ and by Lemma \ref{lem:corres}, ${C}_G(a) \neq {C}_G (b)$.  %It follows that ${C}_G (b)$ is properly contained in $C_G (a)$, and let $\mathfrak{X} = \mathcal {C}(G) - \{ {C}_G (b) \}$.   Note that for each $x \in G \setminus {Z}(G)$ with ${Z}(x) \neq {Z}(b)$, by Lemma \ref{lem:corres}, ${C}_G(x) \neq {C}_G(b)$, and we have ${Z}(x) \leq {C}_G(x) \in \mathfrak{X}$, as ${C}_G(x)$ is properly contained in $C_G (b)$.  For each $x \in G$ with ${Z}(x) = {Z}(b)$, we have ${Z}(x) \subseteq {C}_G(a) \in \mathfrak{X}$. By Proposition \ref{prop:cover_contain_z}, $\mathfrak{X}$ is a cover, and this contradicts the irredundancy of $\mathcal {C}(G)$.
\end{proof}

If $G$ is not an $F$-group, then clearly $\mathcal {C}(G)$ will not be irredundant.   We mention an example of a group $G$ that is not an $F$-group and both maximal subgroups in $\mathcal {C} (G)$ and the minimal subgroups in $\mathcal {C} (G)$ are irredundant as covers.  Let $G = S_4$.  Note that there are $7$ maximal centralizers and they are irredundant as a cover: the four $C_3$ centralizers are abelian, and hence, equal to their centers; and all three $D_8$ maximal centralizers are needed to contain those three centers.  Note that the set of minimal centralizers is irredundant as a cover since all $10$ minimal centralizers are abelian, and hence all are needed to contain their centers (which are themselves).  See Figure \ref{fig:fig1}.  

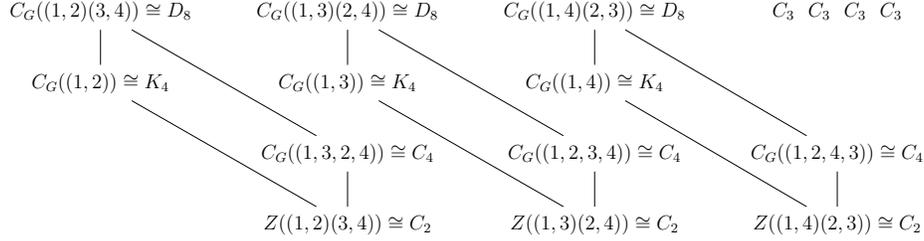
\begin{figure}[H]
	\centering
	\adjustbox{scale=0.64}{%
		\begin{tikzcd}
	{{C}_G((1,2)(3,4))\cong D_8} & {{C}_G((1,3)(2,4))\cong D_8} & {{C}_G((1,4)(2,3))\cong D_8} & {C_3 \,\,\,\, C_3 \,\,\,\, C_3 \,\,\,\, C_3} \\
	{{C}_G((1,2))\cong K_4} & {{C}_G((1,3))\cong K_4} & {{C}_G((1,4))\cong K_4} \\
	& {{C}_G((1,3,2,4))\cong C_4} & {{C}_G((1,2,3,4))\cong C_4} & {{C}_G((1,2,4,3))\cong C_4} \\
	& {Z((1,2)(3,4))\cong C_2} & {Z((1,3)(2,4))\cong C_2} & {Z((1,4)(2,3))\cong C_2}
	\arrow[no head, from=1-1, to=2-1]
	\arrow[no head, from=1-2, to=2-2]
	\arrow[no head, from=1-3, to=2-3]
	\arrow[no head, from=2-1, to=4-2]
	\arrow[no head, from=2-2, to=4-3]
	\arrow[no head, from=2-3, to=4-4]
	\arrow[no head, from=3-2, to=1-1]
	\arrow[no head, from=3-2, to=4-2]
	\arrow[no head, from=3-3, to=1-2]
	\arrow[no head, from=3-3, to=4-3]
	\arrow[no head, from=3-4, to=1-3]
	\arrow[no head, from=3-4, to=4-4]
		\end{tikzcd}
	}
	\caption{A Hasse diagram showing the centralizers and centers of $S_4$.}
	\label{fig:fig1}
\end{figure}

We close the paper with two interesting observations about $F$-groups that are $p$-groups.

\begin{theorem}\label{thm:F-Gp-Exp}
Let $G$ be $p$-group that is an $F$-group and is not a CA-group.  There exists a nonabelian centralizer $C$ with $\mathrm{exp} (C) = \mathrm{exp} ({Z} (C))$. 
\end{theorem}

\begin{proof}
Among all nonabelian proper centralizers in $G$, choose one whose center has the largest exponent, call this centralizer ${C}_G(a)$.  For any element $x \in {C}_G(a) \setminus {Z}(a)$, we have ${Z}(x) < {C}_G(a)$, and as $G$ is an $F$-group, ${Z}(x)$ is not a centralizer, and so, we have a nonabelian proper centralizer ${C}_G(x)$ with center ${Z}(x)$.  Now, $x \in {Z}(x)$, and since ${Z}(a)$ was chosen to have the largest exponent, we have $\mathrm{o}(x) \leq \mathrm{exp}({Z}(a))$.  Since this was true for any element $x \in {C}_G(a) \setminus {Z}(a)$, it follows that $\mathrm{exp}({C}_G(a)) = \mathrm{exp}(Z(a))$. 
\end{proof}

Since semi-extraspecial groups are $F$-groups, our final result is related to Verardi's count of abelian subgroups of maximal possible order in semi-extraspecial groups that is found in Theorem 3.8 of \cite{verardi}.  We close by proving Theorem \ref{intro $F$-group $p$-group}.

%\begin{theorem}
%In a nonabelian $p$-group that is an $F$-group, the number of centralizers is congruent to $1$ modulo $p$.     
%\end{theorem}

\begin{proof}[Proof of Theorem \ref{intro $F$-group $p$-group}]
In an $F$-group, each center $Z = Z^* \cup Z(G)$ where the union is disjoint and $Z^*$ is an equivalence class under the relation $x \bowtie y$ if and only if $Z(x) = Z(y)$.  Thus $|Z^*| = |Z(G)|k$ where $k=|Z : Z(G)| -1$ is congruent to $-1$ modulo $p$.   Now the equivalence classes partition $G$, and so

$$|Z_1^*| + \dots + |Z_n^*| + |Z(G)| = |G|,$$ 
where $n$ is the number of centers/centralizers in $G$, and $Z_1^*, \dots, Z_n^*$ are distinct equivalence classes. Hence $n$ is congruent to $1$ modulo $p$. 
\end{proof}

%\section{}
\bigskip

{\bf Data Availability}:  There is no data associated with this paper.


\begin{thebibliography}{99}
%\bibitem{bertram}  E. A. Bertram, Some applications of graph theory to finite groups, {\it Discrete Math.} {\bf 44} (1983), 31-43.

\bibitem{barg} M. Bhargava, When is a group the union of proper normal subgroups? {\it Amer. Math. Monthly} {\it 109} (2002), 471-473.

\bibitem{bsw} R. Brauer, M. Suzuki, and G. E. Wall,  A characterization of the one-dimensional unimodular projective groups over finite fields, {\it Illinois Journal of Mathematics}, {\bf 2} (1958), 718-745.
	
\bibitem{DHJ} S. Dolfi, M. Herzog and E. Jabara, Finite groups whose noncentral commuting elements have centralizers of equal size, {\it Bull. Aust. Math. Soc.} {\bf 82} (2010), 293-304.
	
%\bibitem{nobound} M. Giudici and C. Parker, There is no upper bound for the diameter of the commuting graph of a finite group, {\it J. Combin. Theory Ser. A} {\bf 120} (2013), 1600-1603. 
	
%\bibitem{symmetric} A. Iranmanesh and A, A. Jafarzadeh, On the commuting graph associated with the symmetric and alternating groups, {\it J. Algebra Appl.} {\bf 7} (2008), 129-146.

\bibitem{Gar} M. Garonzi. Finite Groups That Are the Union of at most 25 Proper Subgroups, {\it J. Algebra Appl.} {\bf 12} (2013), 1350002.

\bibitem{GKS} M. Garonzi, L.-C. Kappe, and E. Swartz, On integers that are covering numbers of groups, {\it Exp. Math.} {\bf 31} (2022), 425-443.

\bibitem{HaAm} S. Haji and S. M. J. Amiri, On groups covered by finitely many centralizers and domination number of the commuting graph, {\it Comm. Algebra} {\bf 47} (2019), 4641-4653.
    
\bibitem{ito} N.  It\^o, On finite groups with given conjugate types I, {\it Nagoya Math. J.} {\bf 6} (1953), 17-28.

\bibitem{cover-exp} L.-C. Kappe, Finite coverings: a journey through groups, loops, rings and semigroups,  Group theory, combinatorics, and computing, 79–88,  {\it Contemp. Math.}, {\bf 611}
American Mathematical Society, Providence, RI, 2014  ISBN:978-0-8218-9435-4
	
\bibitem{max abel} M. L. Lewis, A lower bound on the size of a maximal abelian subgroup, to appear in the Proceedings of Ischia Group Theory Conference 2024.

\bibitem{centgraph} M. L. Lewis and Ryan McCulloch, The commuting graph and a graph associated with centralizers, submitted for publication.   	
https://doi.org/10.48550/arXiv.2511.11926

\bibitem{mason} D. R. Mason, On coverings of a finite group by abelian subgroups, {\it Math. Proc. Cambridge Philos. Soc.} {\bf 83} (1978), 205–209.

\bibitem{reb} J. Rebmann, $F$-Gruppen, {\it Archiv. Math. (Basel)} {\bf 22} (1971), 225-230.

\bibitem{sun} Z.-W. Sun,  Finite covers of groups by cosets or subgroups, {\it Internat. J. Math.} {\bf  17} (2006), 1047-1064.

\bibitem{suz} M. Suzuki,  The nonexistence of a certain type of simple groups of odd order, {\it Proceedings of the American Mathematical Society}, {\bf 8} (1957), 686-695.

\bibitem{verardi} L. Verardi, Gruppi semiextraspeciali di esponente $p$, {\it Ann. Mat. Pura Appl.} {\bf 148} (1987) 131--171.

%\bibitem{d2n} J. Vahidi and A. A. Talebi, The commuting graphs on groups $D_{2n}$ and $Q_n$, {\it J. Math. Comput. Sci.} {\bf 1} (2010), 123-127. 

\bibitem{zappa} G. Zappa, Partitions and other coverings of finite groups, Special issue in honor of Reinhold Baer (1902–1979), {\it Illinois J. Math.} {\bf 47} (2003), 571-580.

\bibitem{zass}  https://www2.math.binghamton.edu/p/zassenhaus/zassenhaus\_2025/program

\end{thebibliography}
\end{document}